\documentclass{article}[11pt]
\usepackage[utf8]{inputenc}
\usepackage{amsthm,amsfonts,amssymb,amsmath,tikz}

\newtheorem{theorem}{Theorem}[section]
\newtheorem{corollary}[theorem]{Corollary}
\newtheorem{lemma}[theorem]{Lemma}
\newtheorem{proposition}[theorem]{Proposition}

\newtheorem{remark}[theorem]{Remark}

\newcommand{\pow}[1]{^{[\natural #1]}}
\newcommand{\dpow}[1]{^{\natural #1}}

\newcommand{\diam}{{\rm diam}}
\newcommand{\rad}{{\rm rad}}
\newcommand{\cp}{\,\square\,}

\begin{document}

\title{Exact distance graphs of product graphs}

\author{
Boštjan Brešar $^{a,b}$
\and
Nicolas Gastineau $^{c}$
\and
Sandi Klav\v zar\footnote{corresponding author} $^{,a,b,d}$
\and
Olivier Togni $^{c}$
}

\date{\today}

\maketitle

\begin{center}
$^a$ Faculty of Natural Sciences and Mathematics, University of Maribor, Slovenia\\
{\tt bostjan.bresar@um.si}
\medskip

$^b$ Institute of Mathematics, Physics and Mechanics, Ljubljana, Slovenia\\
\medskip

$^c$ Laboratoire LE2I, Université de Bourgogne Franche-Comté, France\\
{\tt Nicolas.Gastineau@u-bourgogne.fr} \\
{\tt olivier.togni@u-bourgogne.fr}
\medskip

$^d$ Faculty of Mathematics and Physics, University of Ljubljana, Slovenia\\
{\tt sandi.klavzar@fmf.uni-lj.si}
\end{center}

\begin{abstract}
Given a graph $G$, the exact distance-$p$ graph $G^{[\natural p]}$ has $V(G)$ as its vertex set, and two vertices are adjacent whenever the distance between them in $G$ equals $p$. We present formulas describing the structure of exact distance-$p$ graphs of the Cartesian, the strong, and the lexicographic product. We prove such formulas for the exact distance-$2$ graphs of direct products of graphs. We also consider infinite grids and some other product structures. We characterize the products of graphs of which exact distance graphs are connected. The exact distance-$p$ graphs of hypercubes $Q_n$ are also studied. As these graphs contain generalized Johnson graphs as induced subgraphs, we use some known and find some new constructions of their colorings. These constructions are applied for colorings of the exact distance-$p$ graphs of hypercubes with the focus on the chromatic number of $Q_{n}^{[\natural p]}$ for $p\in \{n-2,n-3,n-4\}$.
\end{abstract}


\noindent\medskip
{\bf AMS Subj.\ Class.\ (2010)}: 05C12, 05C15, 05C76\\

\noindent\medskip
{\bf Running head}: Exact distance graphs of products
\maketitle

\newpage

\baselineskip16pt
\section{Introduction}

Ne\v set\v ril and Ossona de Mendez introduced in~\cite[Section 11.9]{NESET1} the concept of exact distance-$p$ graph, where $p$ is a positive integer, as follows. If $G$ is a graph, then the {\em exact distance-$p$ graph} $G\pow{p}$ of $G$ is the graph with $V(G\pow{p})=V(G)$ and two vertices in $G\pow{p}$ are adjacent if and only if they are at distance exactly $p$ in $G$. Note that $G\pow{1}=G$.  

The main focus in earlier investigations of exact distance graphs was on their chromatic number. One of the main reasons for this interest is the problem asking whether there exists a constant $C$ such that for every odd integer $p$ and every planar graph $G$ we have $\chi(G\pow{p}) \le C$. The problem that was explicitly stated in~\cite[Problem 11.1]{NESET1} and attributed to van den Heuvel and Naserasr (see also~\cite[Problem 1]{NESET2}) has been very recently answered in negative by considering the exact distance graphs of large complete $q$-ary tree~\cite{BOUSQUET}. 
Results on the chromatic number of exact distance graphs are in particular known for trees~\cite{BOUSQUET} and chordal graphs~\cite{QUIROZ}. Also very recently, van den Heuvel, Kierstead and Quiroz~\cite{HEUVEL} proved that for any graph $G$ and odd positive integer $p$, $\chi(G\pow{p})$ is bounded by the weak $(2p-1)$-colouring number of $G$.

The exact distance-$p$ graphs have been much earlier considered for the case when $G$ is a hypercube in the frame of the so-called cube-like graphs~\cite{DVORAK, HARARY, LINIAL, PAYAN, PJWAN, ZIEGLER}, see also the book of Jensen and Toft~\cite{JENSEN}. 
Initially, the notion of the cube-like graph was introduced by Lov\'asz~\cite{HARARY} who proved that every cube-like has integral spectrum. Apparently, many authors had conjectured that the chromatic number of cube-like graphs is always some power of 2. It turned out that there is no cube-like graph of chromatic number $3$ but there exists a cube-like graph of chromatic number $7$~\cite{PAYAN}. 
Ziegler also studied the cube-like graphs (under the name Hamming graphs), and determined the chromatic number in numerous cases.
Finally, the chromatic number of exact distance-$2$ hypercube is a problem which has been intensively studied \cite{LINIAL, PJWAN}.


We believe that the concept of exact distance graphs is not only interesting because of the chromatic number, but also as a general metric graph theory concept. With this paper we thus hope to initiate an interest for general properties of the construction. Actually, using a different language, back in 2001  Ziegler proved the following property for bipartite graphs.

\begin{lemma}{\rm (\cite{ZIEGLER})}
\label{lem:zieg} Let $G$ be a bipartite graph.
\begin{itemize}
\item[(i)] If $p$ is even, then $G\pow{p}$ is not connected.
\item[(ii)] If $p$ is odd, then $G\pow{p}$ is a bipartite graph (and has the same bipartition than $G$).
\end{itemize}
\end{lemma}

In this paper we focus on the exact distance graphs of graph products and proceed as follows. In the rest of this section we give required definitions and fix notation. Then, in Section~\ref{sec:products}, we present formulas describing the structure of exact distance-$p$ graphs of the Cartesian, the strong, and the lexicographic product, respectively, of arbitrary two graphs. In the case of the direct product of graphs only exact distance-$2$ graphs could be expressible with a nice formula, which in turn simplifies to $(G\times H)\pow{2}=G\pow{2}\boxtimes H\pow{2}$ when $G$ and $H$ are both triangle-free graphs. Nice expressions are found also for the exact distance-$2$ graphs of some products of the 2-way infinite paths, which yields the chromatic number of the corresponding grids.
In Section~\ref{sec:connectivity}, we consider the characteristic conditions for the connectivity of exact distance graphs with respect to all four products. This time, for the Cartesian and the direct product we can only deal with the case $p=2$, while for the other two products the result covers exact distance-$p$ graphs for an arbitrary integer $p$. In Section~\ref{sec:hypercubes}, we study the exact distance-$p$ graphs of hypercubes. We start by showing that $Q_{n}\pow{n-1} \cong Q_{n}$, and by describing some structural properties of $Q_{n}\pow{p}$ for an arbitrary $p\le n$. Noting that some generalized Johnson graphs appear as induced subgraphs in $Q_{n}\pow{p}$, we consider the chromatic number of these graphs, combining some results from the literature with some new constructions. This enables us to give upper bounds for the chromatic number of $Q_{n}\pow{p}$ for $p\in \{n-2,n-3,n-4\}$, which are $8$, $15$, and $26$, respectively. 

If $G$ is a graph, then $d_G(x,y)$ is the standard shortest-path distance between vertices $x$ and $y$ in $G$. The maximum distance between $u$ and all the other vertices is the {\em eccentricity} of $u$. The maximum and the minimum eccentricity among the vertices of $G$ are the {\em diameter} $\diam(G)$ and the {\em radius} $\rad(G)$.

 We define $G^{\pow{0}}$ as the graph with the vertex set $V(G)$ and with a loop added to each of its vertices. If $G$ and $H$ are graphs on the same vertex set, then $G\uplus H$ is the graph with vertex set $V(G)=V(H)$ and edge set $E(G) \cup E(H)$. If $G$ is a graph, then $kG$ denotes the disjoint union of $k$ copies of the graph $G$. 

The vertex set of each of the four standard graph products of graphs $G$ and $H$ is equal to $V(G)\times V(H)$. In the \emph{direct product} $G\times H$ vertices $(g_{1},h_{1})$ and $(g_{2},h_{2})$ are adjacent when $g_1g_2\in E(G)$ and $h_1h_2\in E(H)$. In the \emph{lexicographic product} $G\circ H$, vertices $(g_{1},h_{1})$ and $(g_{2},h_{2})$ are adjacent if either $g_{1}g_{2}\in E(G)$, or $g_{1}=g_{2}$ and $h_{1}h_{2}\in E(H)$. In the \emph{strong product} $G\boxtimes H$ vertices $(g_{1},h_{1})$ and $(g_{2},h_{2})$ are adjacent whenever either
$g_1g_2\in E(G)$ and $h_1=h_2$, or $g_1=g_2$ and $h_1h_2\in E(H)$, or $g_1g_2\in E(G)$ and $h_1h_2\in E(H)$.
Finally, in the \emph{Cartesian product} $G\cp H$ vertices $(g_{1},h_{1})$ and $(g_{2},h_{2})$ are adjacent if either
$g_1g_2\in E(G)$ and $h_1=h_2$, or $g_1=g_2$ and $h_1h_2\in E(H)$. All these products are associative and, with the exception of the lexicographic product, also commutative. 
Let $G * H$ be any of the four standard graph products. Then the subgraph of $G * H$
induced by $\{g\} \times V (H)$ is called an {\em $H$-layer} of $G * H$ and denoted $^gH$. For more on products graphs see the book~\cite{HAMMACK}.

\section{Exact distance graphs of graph products}
\label{sec:products}

We first recall the distance function of the four standard products, cf.~\cite{HAMMACK}.

\begin{lemma}{\rm (\cite{HAMMACK})} 
\label{lem:dist}
If $G$ and $H$ are graphs and $(g_1,h_1), (g_2,h_2)\in V(G)\times V(H)$, then
\begin{enumerate}
\item[(i)] $d_{G\cp H}((g_1,h_1),(g_2,h_2))=d_G (g_1,g_2)+d_H (h_1,h_2)$;
\item[(ii)] $d_{G\boxtimes H}((g_1,h_1),(g_2,h_2))=\max\{d_G (g_1,g_2), d_H (h_1,h_2)\}$;
\item[(iii)] $d_{G\times H}((g_1,h_1),(g_2,h_2))=k$, where $k$ is the smallest integer such that there exists a $g_1,g_2$-walk of length $k$ in $G$ and a $h_1,h_2$-walk of length $k$ in $H$;
\item[(iv)] $d_{G\circ H}((g_1,h_1),(g_2,h_2)) = 
\left\{\begin{array}{ll}
d_G (g_1,g_2), & \text{ if } g_1\ne g_2;\\
\min\{d_H (h_1,h_2),2\}, & \text{ if } g_1=g_2 \text{ and } \deg_{G}(g_1)>0;\\
d_H (h_1,h_2), & \text{ otherwise.}\end{array}\right.$
\end{enumerate}
\end{lemma}

\begin{theorem}
\label{thm:cart}
If $G$ and $H$ are graphs, then
$$(G\cp H)\pow{p}=\biguplus_{i=0}^{p} \left(G\pow{i}\times H\pow{p-i}\right)\,.$$
Equivalently, 
$$(G\cp H)\pow{p}=\biguplus_{i=1}^{p-1} \left(G\pow{i}\times H\pow{p-i}\right)\uplus \left(G\pow{p}\cp H\pow{p}\right)\,.$$
\end{theorem}

\begin{proof}
By Lemma~\ref{lem:dist}(i), $d_{G\cp H}((g_1,h_1),(g_2,h_2))=p$ if and only if there exists $i, \ 0\leq i\le p$, such that $d_G (g_1,g_2)=i$ and $d_H (h_1,h_2)=p-i$. This in turn holds if and only if $g_1g_2\in E(G\pow{i})$ and $h_1h_2\in E(G\pow{p-i})$. From this the first equality follows by the definition of the direct product. The second equality follows from the fact that $\left(G\pow{0}\times H\pow{p}\right)\uplus \left(G\pow{p}\times H\pow{0}\right) = G\pow{p}\cp H\pow{p}$. 
\end{proof}

Fig.~\ref{power2cartesian} illustrates Theorem~\ref{thm:cart} on the case $G=P_4$, $H = P_3$, and $p=2$.

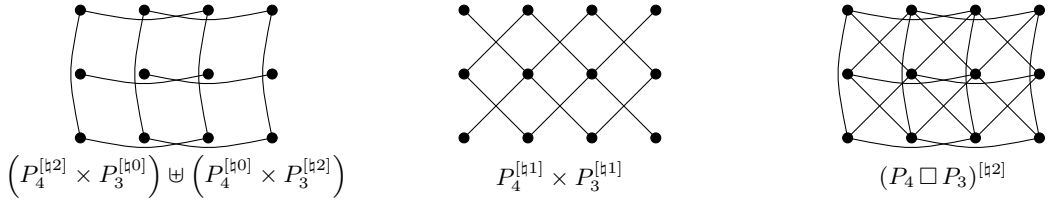
\begin{figure}[ht!]
\begin{center}
\begin{tikzpicture}[scale=0.85]

\draw (0,0) .. controls (-0.2,1) .. (0,2);
\draw (1,0) .. controls (0.8,1) .. (1,2);
\draw (2,0) .. controls (1.8,1) .. (2,2);
\draw (3,0) .. controls (2.8,1) .. (3,2);

\draw (0,0) .. controls (1,-0.2) .. (2,0);
\draw (1,0) .. controls (2,-0.2) .. (3,0);
\draw (0,1) .. controls (1,0.8) .. (2,1);
\draw (1,1) .. controls (2,0.8) .. (3,1);
\draw (0,2) .. controls (1,1.8) .. (2,2);
\draw (1,2) .. controls (2,1.8) .. (3,2);

\node at (0,0)[circle,draw=black,fill=black,scale=0.4]{};
\node at (1,0)[circle,draw=black,fill=black,scale=0.4]{};
\node at (2,0)[circle,draw=black,fill=black,scale=0.4]{};
\node at (3,0)[circle,draw=black,fill=black,scale=0.4]{};
\node at (0,1)[circle,draw=black,fill=black,scale=0.4]{};
\node at (1,1)[circle,draw=black,fill=black,scale=0.4]{};
\node at (2,1)[circle,draw=black,fill=black,scale=0.4]{};
\node at (3,1)[circle,draw=black,fill=black,scale=0.4]{};
\node at (0,2)[circle,draw=black,fill=black,scale=0.4]{};
\node at (1,2)[circle,draw=black,fill=black,scale=0.4]{};
\node at (2,2)[circle,draw=black,fill=black,scale=0.4]{};
\node at (3,2)[circle,draw=black,fill=black,scale=0.4]{};

\node at (1.5,-0.6){\small{$ \left(P_4\pow{2}\times P_3\pow{0}\right)\uplus \left(P_4\pow{0}\times P_3\pow{2}\right) $}};

\draw (0+6,0) -- (2+6,2);
\draw (1+6,0) -- (3+6,2);
\draw (0+6,1) -- (1+6,2);
\draw (2+6,0) -- (3+6,1);

\draw (3+6,0) -- (1+6,2);
\draw (2+6,0) -- (0+6,2);
\draw (3+6,1) -- (2+6,2);
\draw (1+6,0) -- (0+6,1);

\node at (0+6,0)[circle,draw=black,fill=black,scale=0.4]{};
\node at (1+6,0)[circle,draw=black,fill=black,scale=0.4]{};
\node at (2+6,0)[circle,draw=black,fill=black,scale=0.4]{};
\node at (3+6,0)[circle,draw=black,fill=black,scale=0.4]{};
\node at (0+6,1)[circle,draw=black,fill=black,scale=0.4]{};
\node at (1+6,1)[circle,draw=black,fill=black,scale=0.4]{};
\node at (2+6,1)[circle,draw=black,fill=black,scale=0.4]{};
\node at (3+6,1)[circle,draw=black,fill=black,scale=0.4]{};
\node at (0+6,2)[circle,draw=black,fill=black,scale=0.4]{};
\node at (1+6,2)[circle,draw=black,fill=black,scale=0.4]{};
\node at (2+6,2)[circle,draw=black,fill=black,scale=0.4]{};
\node at (3+6,2)[circle,draw=black,fill=black,scale=0.4]{};

\node at (1.5+6,-0.6){ \small{$P_4\pow{1}\times P_3\pow{1} $}};

\draw (0+12,0) -- (2+12,2);
\draw (1+12,0) -- (3+12,2);
\draw (0+12,1) -- (1+12,2);
\draw (2+12,0) -- (3+12,1);

\draw (3+12,0) -- (1+12,2);
\draw (2+12,0) -- (0+12,2);
\draw (3+12,1) -- (2+12,2);
\draw (1+12,0) -- (0+12,1);

\draw (0+12,0) .. controls (-0.2+12,1) .. (0+12,2);
\draw (1+12,0) .. controls (0.8+12,1) .. (1+12,2);
\draw (2+12,0) .. controls (1.8+12,1) .. (2+12,2);
\draw (3+12,0) .. controls (2.8+12,1) .. (3+12,2);

\draw (0+12,0) .. controls (1+12,-0.2) .. (2+12,0);
\draw (1+12,0) .. controls (2+12,-0.2) .. (3+12,0);
\draw (0+12,1) .. controls (1+12,0.8) .. (2+12,1);
\draw (1+12,1) .. controls (2+12,0.8) .. (3+12,1);
\draw (0+12,2) .. controls (1+12,1.8) .. (2+12,2);
\draw (1+12,2) .. controls (2+12,1.8) .. (3+12,2);

\node at (0+12,0)[circle,draw=black,fill=black,scale=0.4]{};
\node at (1+12,0)[circle,draw=black,fill=black,scale=0.4]{};
\node at (2+12,0)[circle,draw=black,fill=black,scale=0.4]{};
\node at (3+12,0)[circle,draw=black,fill=black,scale=0.4]{};
\node at (0+12,1)[circle,draw=black,fill=black,scale=0.4]{};
\node at (1+12,1)[circle,draw=black,fill=black,scale=0.4]{};
\node at (2+12,1)[circle,draw=black,fill=black,scale=0.4]{};
\node at (3+12,1)[circle,draw=black,fill=black,scale=0.4]{};
\node at (0+12,2)[circle,draw=black,fill=black,scale=0.4]{};
\node at (1+12,2)[circle,draw=black,fill=black,scale=0.4]{};
\node at (2+12,2)[circle,draw=black,fill=black,scale=0.4]{};
\node at (3+12,2)[circle,draw=black,fill=black,scale=0.4]{};

\node at (1.5+12,-0.6){ \small{$(P_4\cp P_3)\pow{2} $}};
\end{tikzpicture}
\end{center}
\caption{Illustration of the structure of $(P_4\cp P_3)\pow{2} $ which is isomorphic to $\left(P_4\pow{2}\times P_3\pow{0}\right)\uplus \left(P_4\pow{0}\times P_3\pow{2}\right) \uplus \left( P_4\pow{1}\times P_3\pow{1} \right)$.}
\label{power2cartesian}
\end{figure}

\begin{theorem}
\label{thm:strong}
If $G$ and $H$ are graphs, then
$$(G\boxtimes H)\pow{p}=\biguplus_{i=0}^{p} \left((G\pow{p}\times H\pow{i})\uplus (G\pow{i}\times H\pow{p})\right)\,.$$ 
\end{theorem}

\begin{proof}
By Lemma~\ref{lem:dist}(ii), $d_{G\boxtimes H}((g_1,h_1),(g_2,h_2))=p$ if and only if 
either $d_G (g_1,g_2)=p$ and $d_H (h_1,h_2)=i$, where $0\le i\le p$, or $d_G (g_1,g_2)=i$ and $d_H (h_1,h_2)=p$, where $0\le i\le p$. Hence, the theorem follows.
\end{proof}

In view of Lemma~\ref{lem:dist}(iii), it is not surprising that the situation with the direct product is more tricky (as it is often the case with the direct product). To state a formula for $(G\times H)\pow{2}$, we need the following concept, see~\cite[Section 11.9]{NESET1}. If $G$ is a graph, then $G\dpow{p}$ is the graph with $V(G\dpow{p})=V(G)$, vertices $x$ and $y$ being adjacent if and only if they are connected in $G$ with a path of length $p$. 

\begin{theorem}
\label{thm:direct}
If $G$ and $H$ are graphs without isolated vertices, then 
$$(G\times H)\pow{2}=(G\dpow{2}\square H\dpow{2} )\uplus (G\dpow{2}\times H\pow{2} )\uplus (G\pow{2}\times H\dpow{2})\,.$$
In particular, if $G$ and $H$ are triangle-free, then
$$(G\times H)\pow{2}=G\pow{2}\boxtimes H\pow{2}\,.$$
\end{theorem}

\begin{proof}
Let $(g_1,h_1), (g_2,h_2)$ be vertices of $G\times H$ with $d_{G\times H}((g_1,h_1),(g_2,h_2))=2$. Then by Lemma~\ref{lem:dist}(iii) there exist a $g_1,g_2$-walk of length $2$ in $G$ and a $h_1,h_2$-walk of length $2$ in $H$, and not both $g_1g_2\in E(G)$ and $h_1h_2\in E(H)$ hold.

If $g_1=g_2$, then, $d_{G\times H}((g_1,h_1),(g_2,h_2))=2$ if and only if there is a path of length 2 between $h_1$ and $h_2$ in $H$. Note that the sufficiency of this assertion holds because $G$ is isolate-free. Similarly, if $h_1=h_2$, then, $d_{G\times H}((g_1,h_1),(g_2,h_2))=2$ if and only if there is a path of length 2 between $g_1$ and $g_2$ in $G$, where we use the fact that $H$ is isolate-free. It follows that $G\dpow{2}\square H\dpow{2}$ is a spanning subgraph of $(G\times H)\pow{2}$.

Suppose next that $g_1\neq g_2$ and $h_1\neq h_2$. Then $d_{G\times H}((g_1,h_1),(g_2,h_2))=2$ if and only if 
\begin{itemize}
\item either there is a path of length 2 between $h_1$ and $h_2$ in $H$ and $d_{G}(g_1,g_2)=2$, 
\item or there is a path of length 2 between $g_1$ and $g_2$ in $G$ and $d_{H}(h_1,h_2)=2$.  
\end{itemize}
(Indeed, if $g_1g_2\in E(G)$ and $h_1h_2\in E(H)$, then $(g_1,h_1)(g_2,h_2)\in E(G\times H)$, and if there is no $g_1,g_2$-path of length $2$ in $G$ or no $h_1, h_2$-path of length $2$ in $H$, then $d_{G\times H}((g_1,h_1),(g_2,h_2))>2$.) The first possibility implies that $G\dpow{2}\times H\pow{2}$ is a spanning subgraph of $(G\times H)\pow{2}$, while the second possibility implies that same for $G\pow{2}\times H\dpow{2}$. This proves the first formula of the theorem. 

Suppose now that $G$ and $H$ are triangle-free. Then $G\pow{2} = G\dpow{2}$ and $H\pow{2} = H\dpow{2}$. By the already proved formula we have 
\begin{eqnarray*}
(G\times H)\pow{2} & = & (G\dpow{2}\cp H\dpow{2} )\uplus (G\dpow{2}\times H\pow{2} )\uplus (G\pow{2}\times H\dpow{2})  \\
& = & (G\pow{2}\cp H\pow{2} )\uplus (G\pow{2}\times H\pow{2}) \\
& = & G\pow{2}\boxtimes H\pow{2}\,,
\end{eqnarray*}
where the last equality holds by the basic relation between the three products in question.
\end{proof}

For the lexicographic product, the case where $G$ is trivial is special since we have $(K_1\circ H)\pow{p}=H\pow{p}$. If $G$ has no isolated vertex, we have the following.

\begin{theorem}
\label{thm:lexicographic}
If $G$ is a graph without isolated vertices and $H$ an arbitrary graph, then 
$$(G\circ H)\pow{p}=\left\{\begin{array}{ll}
G\pow{2}\circ \overline{H}, &\text{if } p=2;\\
G\pow{p}\circ \overline{K}_{n(H)}, & \text{otherwise}.\end{array}\right.$$
\end{theorem}

\begin{proof}
By Lemma~\ref{lem:dist}, $d_{G\circ H}((g_1,h_1),(g_2,h_2))=\min\{d_H (h_1,h_2),2\}$ if $g_1=g_2$ or $d_G (g_1,g_2)$, otherwise. 
First, if $p=2$, then two vertices $(g,h_1)$ and $(g,h_2)$ are at distance two in $G\circ H$ if and only if $h_1\neq h_2$ and they are not adjacent. 
Also, vertices $(g_1,h_1)$ and $(g_2,h_2)$, where $g_1\neq g_2$, are at distance $2$ if and only if $d_G(g_1,g_2)=2$. Consequently, $(G\circ H)\pow{2}=G\pow{2}\circ \overline{H}$.  

Second, if $p\ge3$, then no vertices $(g,h_1)$ and $(g,h_2)$ are adjacent in $(G\circ H)\pow{p}$.
Also, vertices  $(g_1,h_1)$ and $(g_2,h_2)$, where $g_1\neq g_2$, are at distance $p$ if and only if $d_G(g_1,g_2)=p$. Consequently, $(G\circ H)\pow{p}=G\pow{p}\circ \overline{K}_{n(H)}$.  
\end{proof}

We now turn to infinite graphs and state the following interesting representations of exact distance-2 graphs of infinite grids. 

\begin{proposition}
\label{prp:infinite-grids}
If $P_\infty$ is the $2$-way infinite path, then 
\begin{itemize}
\item[(1)] $(P_\infty \cp P_\infty)\pow{2} = 2(P_\infty \boxtimes P_\infty)$, and
\item[(2)] $(P_\infty \times P_\infty)\pow{2} = 4(P_\infty \boxtimes P_\infty)$.
\end{itemize}
\end{proposition}

\begin{proof}
Throughout the proof let $V(P_\infty) = \mathbb{Z}$, so that the vertex set of each of the products considered as well as of their distance-2 graphs is $\mathbb{Z} \times \mathbb{Z}$. 

(1) A vertex $(i,j)\in \mathbb{Z} \times \mathbb{Z}$ of $P_\infty \cp P_\infty$ is adjacent to the four vertices $(i,j\pm 1)$ and $(i\pm 1, j)$. Consequently, in $(P_\infty \cp P_\infty)\pow{2}$, the vertex $(i,j)$ is adjacent to the vertices $(i\pm 1, j\pm 1)$, $(i,j\pm 2)$, and $(i\pm 2, j)$. (Note that $(P_\infty \cp P_\infty)\pow{2}$ is $8$-regular). It follows that $(P_\infty \cp P_\infty)\pow{2}$ consists of two connected components, one component being induced by the vertices $(i,j)$ such that $i+j$ is even, and the other component being induced by the vertices $(i,j)$ such that $i+j$ is odd. Let these components be called {\em even} and {\em odd}, respectively. 

Consider the even component of $(P_\infty \cp P_\infty)\pow{2}$ and for $k\in {\mathbb Z}$ set $V_k = \{(i,j):\ i+j = 2k\}$. A vertex from $V_k$ has two neighbors in $V_k$, and three neighbors in each of $V_{k-1}$ and $V_{k+1}$. Hence the set $V_k$ induces a subgraph isomorphic to $P_\infty$ and, moreover, $V_k\cup V_{k+1}$ (as well as $V_k\cup V_{k-1}$) induces a subgraph isomorphic to $P_2\boxtimes P_\infty$. This fact is illustrated in Fig.~\ref{fig:strong-layers} for $k=0$, that is, for the sets $V_0$, $V_1$, and $V_{-1}$. 

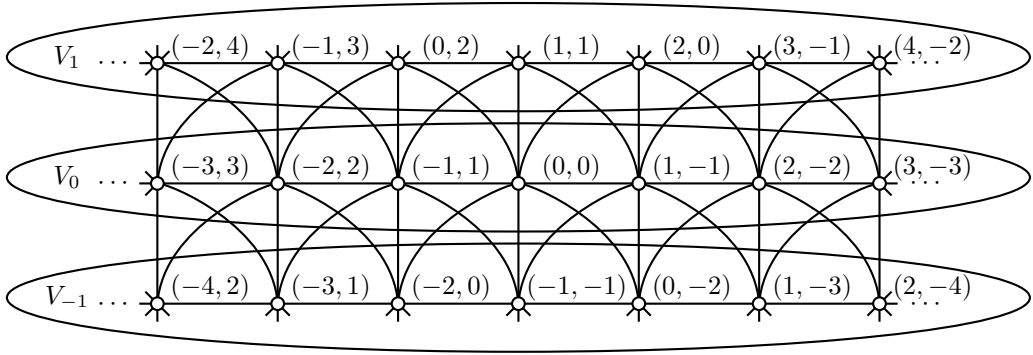
\begin{figure}[ht!]
\begin{center}
\begin{tikzpicture}[scale=0.8,style=thick]
\def\vr{3pt}
\path (0,0) coordinate (start);

\foreach \i in {1,2,...,7}
{
    \path (start) + (\i*2,0) coordinate (x\i);
    \path (start) + (\i*2,2) coordinate (y\i);
    \path (start) + (\i*2,4) coordinate (z\i);
}
\foreach \i in {1,2,...,7}
  {\draw (x\i) -- (y\i) -- (z\i);}
\draw (x1) -- (x2) -- (x3) -- (x4) --(x5) -- (x6) -- (x7); 
\draw (y1) -- (y2) -- (y3) -- (y4) --(y5) -- (y6) -- (y7); 
\draw (z1) -- (z2) -- (z3) -- (z4) --(z5) -- (z6) -- (z7); 
\draw (x1) .. controls (2.1,1.3) and (3.8,2) .. (y2);
\draw (x2) .. controls (3.9,1.3) and (2.2,2) .. (y1);
\draw (x2) .. controls (4.1,1.3) and (5.8,2) .. (y3);
\draw (x3) .. controls (5.9,1.3) and (4.2,2) .. (y2);
\draw (x3) .. controls (6.1,1.3) and (7.8,2) .. (y4);
\draw (x4) .. controls (7.9,1.3) and (6.2,2) .. (y3);
\draw (x4) .. controls (8.1,1.3) and (9.8,2) .. (y5);
\draw (x5) .. controls (9.9,1.3) and (8.2,2) .. (y4);
\draw (x5) .. controls (10.1,1.3) and (11.8,2) .. (y6);
\draw (x6) .. controls (11.9,1.3) and (10.2,2) .. (y5);
\draw (x6) .. controls (12.1,1.3) and (13.8,2) .. (y7);
\draw (x7) .. controls (13.9,1.3) and (12.2,2) .. (y6);
\draw (y1) .. controls (2.1,3.3) and (3.8,4) .. (z2);
\draw (y2) .. controls (3.9,3.3) and (2.2,4) .. (z1);
\draw (y2) .. controls (4.1,3.3) and (5.8,4) .. (z3);
\draw (y3) .. controls (5.9,3.3) and (4.2,4) .. (z2);
\draw (y3) .. controls (6.1,3.3) and (7.8,4) .. (z4);
\draw (y4) .. controls (7.9,3.3) and (6.2,4) .. (z3);
\draw (y4) .. controls (8.1,3.3) and (9.8,4) .. (z5);
\draw (y5) .. controls (9.9,3.3) and (8.2,4) .. (z4);
\draw (y5) .. controls (10.1,3.3) and (11.8,4) .. (z6);
\draw (y6) .. controls (11.9,3.3) and (10.2,4) .. (z5);
\draw (y6) .. controls (12.1,3.3) and (13.8,4) .. (z7);
\draw (y7) .. controls (13.9,3.3) and (12.2,4) .. (z6);
\draw (z1) -- ++(0.2,0.2); \draw (z1) -- ++(0.0,0.3); 
\draw (z1) -- ++(-0.2,0.2); \draw (z1) -- ++(-0.3,0.0);
\draw (z1) -- ++(-0.2,-0.2);
\draw (z7) -- ++(0.2,0.2); \draw (z7) -- ++(0.0,0.3); 
\draw (z7) -- ++(-0.2,0.2); \draw (z7) -- ++(0.3,0.0);
\draw (z7) -- ++(0.2,-0.2);
\draw (z2) -- ++(-0.2,0.2); \draw (z2) -- ++(0,0.3);
\draw (z2) -- ++(0.2,0.2);
\draw (z3) -- ++(-0.2,0.2); \draw (z3) -- ++(0,0.3);
\draw (z3) -- ++(0.2,0.2);
\draw (z4) -- ++(-0.2,0.2); \draw (z4) -- ++(0,0.3);
\draw (z4) -- ++(0.2,0.2);
\draw (z5) -- ++(0.2,0.2);
\draw (z5) -- ++(-0.2,0.2); \draw (z5) -- ++(0,0.3);
\draw (z5) -- ++(0.2,0.2);
\draw (z6) -- ++(-0.2,0.2); \draw (z6) -- ++(0,0.3);
\draw (z6) -- ++(0.2,0.2);
\draw (y1) -- ++(-0.2,0.2); \draw (y1) -- ++(-0.3,0.0);
\draw (y1) -- ++(-0.2,-0.2);
\draw (y7) -- ++(0.2,0.2); \draw (y7) -- ++(0.3,0.0);
\draw (y7) -- ++(0.2,-0.2);
\draw (x1) -- ++(-0.2,0.2); \draw (x1) -- ++(-0.3,0.0); 
\draw (x1) -- ++(-0.2,-0.2); \draw (x1) -- ++(0,-0.3);
\draw (x1) -- ++(0.2,-0.2);
\draw (x7) -- ++(0.2,0.2); \draw (x7) -- ++(0.3,0.0); 
\draw (x7) -- ++(0.2,-0.2); \draw (x7) -- ++(0,-0.3);
\draw (x7) -- ++(-0.2,-0.2);
\draw (x2) -- ++(-0.2,-0.2); \draw (x2) -- ++(0,-0.3);
\draw (x2) -- ++(0.2,-0.2);
\draw (x3) -- ++(-0.2,-0.2); \draw (x3) -- ++(0,-0.3);
\draw (x3) -- ++(0.2,-0.2);
\draw (x4) -- ++(-0.2,-0.2); \draw (x4) -- ++(0,-0.3);
\draw (x4) -- ++(0.2,-0.2);
\draw (x5) -- ++(-0.2,-0.2); \draw (x5) -- ++(0,-0.3);
\draw (x5) -- ++(0.2,-0.2);
\draw (x6) -- ++(-0.2,-0.2); \draw (x6) -- ++(0,-0.3);
\draw (x6) -- ++(0.2,-0.2);
\foreach \i in {1,2,...,7}
{
    \draw (x\i) [fill=white] circle (\vr);   
    \draw (y\i) [fill=white] circle (\vr);   
    \draw (z\i) [fill=white] circle (\vr);   
}
\draw (x1) node[xshift=20pt, yshift=6pt]  {$(-4,2)$};
\draw (x2) node[xshift=20pt, yshift=6pt]  {$(-3,1)$};
\draw (x3) node[xshift=20pt, yshift=6pt]  {$(-2,0)$};
\draw (x4) node[xshift=22pt, yshift=6pt]  {$(-1,-1)$};
\draw (x5) node[xshift=20pt, yshift=6pt]  {$(0,-2)$};
\draw (x6) node[xshift=20pt, yshift=6pt]  {$(1,-3)$};
\draw (x7) node[xshift=20pt, yshift=6pt]  {$(2,-4)$};
\draw (y1) node[xshift=20pt, yshift=6pt]  {$(-3,3)$};
\draw (y2) node[xshift=20pt, yshift=6pt]  {$(-2,2)$};
\draw (y3) node[xshift=20pt, yshift=6pt]  {$(-1,1)$};
\draw (y4) node[xshift=20pt, yshift=6pt]  {$(0,0)$};
\draw (y5) node[xshift=20pt, yshift=6pt]  {$(1,-1)$};
\draw (y6) node[xshift=20pt, yshift=6pt]  {$(2,-2)$};
\draw (y7) node[xshift=20pt, yshift=6pt]  {$(3,-3)$};
\draw (z1) node[xshift=20pt, yshift=6pt]  {$(-2,4)$};
\draw (z2) node[xshift=20pt, yshift=6pt]  {$(-1,3)$};
\draw (z3) node[xshift=20pt, yshift=6pt]  {$(0,2)$};
\draw (z4) node[xshift=20pt, yshift=6pt]  {$(1,1)$};
\draw (z5) node[xshift=20pt, yshift=6pt]  {$(2,0)$};
\draw (z6) node[xshift=20pt, yshift=6pt]  {$(3,-1)$};
\draw (z7) node[xshift=20pt, yshift=6pt]  {$(4,-2)$};
\draw (0.5,4.1) node {$V_{1}$};
\draw (0.5,2.1) node {$V_{0}$};
\draw (0.5,0.1) node {$V_{-1}$};
\draw (1.3,0) node {$\cdots$};
\draw (1.3,2) node {$\cdots$};
\draw (1.3,4) node {$\cdots$};
\draw (14.8,0) node {$\cdots$};
\draw (14.8,2) node {$\cdots$};
\draw (14.8,4) node {$\cdots$};

\draw (8,0.1) ellipse (8.5cm and 0.9cm);
\draw (8,2.1) ellipse (8.5cm and 0.9cm);
\draw (8,4.1) ellipse (8.5cm and 0.9cm);
\end{tikzpicture}
\end{center}
\caption{Central parts of the sets $V_0$, $V_1$, and $V_{-1}$ of the even component of $(P_\infty \cp P_\infty)\pow{2}$}
\label{fig:strong-layers}
\end{figure}

By the above local strong product structure induced by the sets $V_{k}\cup V_{k+1}$, $k\in {\mathbb Z}$, we inductively conclude that the even component of $(P_\infty \cp P_\infty)\pow{2}$ is isomorphic to $P_\infty \boxtimes P_\infty$. A parallel argument applies to the odd component. This proves the first assertion of the proposition. 

\medskip
(2) A vertex $(i,j)\in \mathbb{Z} \times \mathbb{Z}$ of $P_\infty \times P_\infty$ is adjacent to the vertices $(i\pm 1,j\pm 1)$ and consequently the vertex $(i,j)$ of $(P_\infty \times P_\infty)\pow{2}$ is adjacent to the vertices $(i\pm 2, j\pm 2)$, $(i,j\pm 2)$, and $(i\pm 2, j)$. Let $X_{00} = \{(i,j):\ i,j\ {\rm even}\}$, 
$X_{01} = \{(i,j):\ i\ {\rm even}, j\ {\rm odd}\}$, 
$X_{10} = \{(i,j):\ i\ {\rm odd}, j\ {\rm even}\}$, and
$X_{11} = \{(i,j):\ i,j\ {\rm odd}\}$. Then $(P_\infty \times P_\infty)\pow{2}$ consists of four connected components $G_{ij}$, $i,j\in \{0,1\}$, where $G_{ij}$ is induced by the vertex set $X_{ij}$. It is straightforward to see that each of the $G_{ij}$ induces a subgraph of $(P_\infty \times P_\infty)\pow{2}$ isomorphic to $P_\infty \boxtimes P_\infty$, hence the second assertion of the proposition. 
\end{proof}

Formula (2) of the above proposition could also be proven in the following way. One should first observe (and prove) that the direct product $P_\infty \times P_\infty$ is isomorphic to the disjoint union of two copies of the square grid $P_\infty \cp P_\infty$, and then apply Proposition~\ref{prp:infinite-grids}(1). 

Note that in view of Proposition~\ref{prp:infinite-grids} it is obvious that $$\chi((P_\infty \cp P_\infty)\pow{2}) = \chi((P_\infty \times P_\infty)\pow{2}) = 4.$$
The graph $(P_\infty \boxtimes P_\infty)\pow{2}$ is $16$-regular, but its structure is not so transparent. Nevertheless, $\chi((P_\infty \boxtimes P_\infty)\pow{2}) = 4$ as can be demonstrated by first coloring the vertices $(i,j)$, $i, j\in \{1,2,3,4\}$, with the following pattern:

\medskip
\begin{tabular}{ c c c c c c c c}
  1 & 1 & 2 & 2 \\
  1 & 1 & 2 & 2 \\
  3 & 3 & 4 & 4 \\
  3 & 3 & 4 & 4 \\ 
\end{tabular}

\medskip\noindent
and then repeatedly extending the pattern to the whole graph $(P_\infty \boxtimes P_\infty)\pow{2}$. This pattern can be generalized to an arbitrary $p\ge 1$ to get a 4-coloring of $(P_\infty \boxtimes P_\infty)\pow{p}$ as follows:

\medskip
\begin{tabular}{ c c c c c c c c}
  1 & \ldots & 1 & 2 & \ldots & 2 \\
  \vdots &   &  \vdots  & \vdots   &        & \vdots \\ 
  1 & \ldots & 1 & 2 & \ldots & 2 \\
  3 & \ldots & 3 & 4 & \ldots & 4 \\
   \vdots &   &  \vdots  & \vdots   &        & \vdots \\ 
  3 & \ldots & 3 & 4 & \ldots & 4 \\ 
\end{tabular}

\medskip\noindent 
Hence, we infer that 
$$\chi((P_\infty \boxtimes P_\infty)\pow{p}) = 4$$ 
holds for every positive integer $p$. 

It seems intriguing to find a nice expression for $(P_\infty \cp P_\infty)\pow{p}$ and $(P_\infty \times P_\infty)\pow{p}$ when $p>2$.


\section{Connectivity}
\label{sec:connectivity}

We start with the following easy observation.

\begin{lemma}
\label{lem:rad-con}
If $G$ is a non-trivial graph and $p>\rad(G)$, then $G\pow{p}$ is not connected.
\end{lemma}
Indeed, if $p>\rad(G)$, then every vertex whose eccentricity equals $\rad(G)$ is an isolated vertex of $G\pow{p}$.

\begin{theorem}
Let $G$ and $H$ be connected graphs with $\rad(G)\ge \rad(H)$, and let $p\ge 2$. The graph $(G\boxtimes H)\pow{p}$ is connected if and only if the following conditions hold:
\begin{itemize}
\item[(1)] $\rad(G)\ge p$, and
\item[(2)]$G\pow{p}$ is connected or $\diam(H)\ge p$.
\end{itemize}
\end{theorem}

\begin{proof}
First, suppose that $(G\boxtimes H)\pow{p}$ is connected. Since $\rad(G\boxtimes H)=\max\{\rad(G),\rad(H)\}=\rad(G)$, it follows, by applying Lemma~\ref{lem:rad-con}, that $\rad(G)\ge p$. Suppose next that condition {\em (2)} does not hold, that is, $G\pow{p}$ is not connected and $\diam(H)< p$.  Let $P$ be a shortest path between $(g,h)$ and $(g',h')$ of length $p$ in $G\boxtimes H$. Then, since $\diam(H)<p$, the projection of $P$ on $G$ is a (shortest) $g,g'$-path of length $p$ in $G$. 
In other words, starting from a vertex $(g,h)$ one can reach by shortest paths of length $p$ in $G\boxtimes H$ only the vertices in the layers $^{g'}\!H$, where $d_G(g,g')=p$. Hence, if $g_1$ and $g_2$ are vertices that belong to different connected components of $G\pow{p}$, and $h$ is an arbitrary vertex of $H$, then $(g_1,h)$ and $(g_2,h)$ belong to different connected components of $(G\boxtimes H)\pow{p}$.

For the converse, assume that conditions {\em (1)} and {\em (2)} hold. We distinguish two cases.

In the first case, suppose that $\rad(G)\ge p$, and $\diam(H)\ge p$. Let $(g,h)$ be a vertex in $G\boxtimes H$. In the same way as in the first paragraph we can show that all the vertices in $^{g}\!H$ are in the same connected component of $(G\boxtimes H)\pow{p}$. Let $g'g\in E(G)$ for some $g'\in V(G)$, and let $h$ and $h'$ be vertices in $H$ at distance $p$ (as $p\le \diam(H)$ such two vertices exist). Let $h''$ be a neighbor of $h$ that lies on a shortest $h,h'$-path. Note that $d_{G\boxtimes H}[(g,h),(g',h')]=p$, where the neighbor on a shortest $(g,h),(g',h')$-path is $(g,h'')$. Hence $(g',h')$ is in the same connected component of $(G\boxtimes H)\pow{p}$ as all the vertices of $^{g}\!H$. By the same reasoning as before, all vertices in $^{g}\!H$ are in the same component as $(g',h)$ of $(G\boxtimes H)\pow{p}$. As $G$ is connected, an inductive argument implies that all $H$-layers are in one and the same component. 

In the second case, let $G\pow{p}$ be connected (and $\rad(G)\ge p$). By excluding the first case, suppose moreover that $\diam(H)<p$.
Let $(g,h)$ be a vertex in $G\boxtimes H$, and let $g'\in V(G)$ be a neighbor of $g$ in $G\pow{p}$. Hence, all vertices from the layer $^{g'}\!H$ are adjacent in $(G\boxtimes H)\pow{p}$ to the vertex $(g,h)$. In turn, by reversing the roles of $g$ and $g'$, all vertices in $^{g}\!H$ is adjacent to all vertices from the layer $^{g'}\!H$. Since $G\pow{p}$ is connected, an inductive arguments yields that $(G\boxtimes H)\pow{p}$ is connected. 
\end{proof}

The situation of the lexicographic product is the following.

\begin{proposition}
If $p\ge 1$ and $G$ is a non-trivial graph, then $(G\circ H)\pow{p}$ is connected if and only if $G\pow{p}$ is connected.
\end{proposition}

\begin{proof}
The assertion for $p=1$ follows, since $G\circ H$ is connected if and only if $G$ is connected. 

Let $p=2$. Suppose that $G\pow{2}$ is not connected. Note that any shortest path of length $2$ from a vertex in $^{g}\!H$ either ends in the same layer or in a layer $^{g'}\!H$, where $gg'\in G\circ H\pow{2}$. Hence $G\pow{2}$ is not connected, implies that $(G\circ H)\pow{p}$ is not connected., Assume conversely that $G\pow{2}$ is connected. Let $(g,h)$ and $(g,h')$ be arbitrary vertices from $^{g}\!H$. If $hh'\notin E(H)$, then $d_{G\circ H}[(g,h),(g,h')]=2$, which implies that $(g,h)$ and $(g,h')$ are in the same component of $G\pow{2}$. Now, let $hh'\in E(H)$. Since $G\pow{2}$ is connected, $\rad(G)\ge 2$. Hence, there exists a vertex $g'\in V(G)$ with $d_G(g,g')=2$. Then, $d_{G\circ H}[(g,h),(g',h)]=d_{G\circ H}[(g',h),(g,h')]=2$, which implies that $(g,h)$ and $(g,h')$ are in the same component of $(G\circ H)\pow{p}$. The above two cases imply that all vertices from $^{g}\!H$ are in the same component of $(G\circ H)\pow{p}$. 
Because $G\pow{p}$ is connected, inductive argument yields that $(G\circ H)\pow{p}$ is connected. 

Finally, let $p\ge 3$. The projection to $G$ of any shortest path in $G\circ H$ of length $p$ is a shortest path in $G$ of the same length. From this observation the assertion follows immediately.
\end{proof}

\begin{proposition}
Let $G$ and $H$ be connected graphs. Then,
\begin{itemize}
\item[(a)] $(G\cp H)\pow{2}$ is connected if and only if one of $G$ or $H$ is non-bipartite.
\item[(b)] $(G\times H)\pow{2}$ is connected if and only if $G\pow{2}$ and $H\pow{2}$ are connected.
\end{itemize}
\end{proposition}

\begin{proof} 
(a) Theorem~\ref{thm:cart} implies that $G\times H$ is a spanning subgraph of $(G\cp H)\pow{2}$. The result now follows by Weichel's theorem~\cite{WEICHSEL} asserting that $G\times H$ is connected if and only if $G$ and $H$ are connected and at least one of them is not bipartite. 

(b) By Theorem~\ref{thm:direct}, $(G\times H)\pow{2}$ contains $G\pow{2}\boxtimes H\pow{2}$ as a spanning subgraph (because $E(G\pow{2})\subseteq E(G\dpow{2})$, for any graph $G$).
The claim now follows because the strong product is connected if and only if both factor graphs are connected, see~\cite{HAMMACK}.
\end{proof}

The connectivity of $(G\cp H)\pow{p}$ and of $(G\times H)\pow{p}$ where $p\ge 3$ seems an intriguing open question. In the next result we solve it for the particular case of {\em hypercubes} $Q_d$, where $Q_1 = K_2$ and $Q_d = Q_{d-1}\cp K_2$ for $d\ge 2$. 

\begin{theorem}
\label{thm:hypercubes-connectivity}
Let $d\ge 2$ and $1\le p < d$. Then $Q_d\pow{p}$ is connected if and only if $p$ is odd.  
\end{theorem}

\begin{proof}
If $p$ is even, then $Q_d\pow{p}$ is disconnected by Lemma~\ref{lem:zieg}(a).

Assume now that $p$ is odd. The case $p=1$ is trivial, hence assume in the rest that $p\ge 3$. To prove that $Q_d\pow{p}$ is connected it suffices to show that in $Q_d\pow{p}$ there exists a path from the vertex $0^d$ to a vertex with exactly one bit $1$. Indeed, if this is proved, then since $Q_d$ is edge-transitive, every pair of adjacent vertices of $Q_d$ is connected by a path in $Q_d\pow{p}$. Consequently, as $Q_d$ is connected, $Q_d\pow{p}$ is also connected.

Clearly, the vertex $x_0 = 0^d$ is adjacent in $Q_d\pow{p}$ to the vertex $x_1 = 1^p0^{d-p}$, which is in turn adjacent to the vertex $x_2 = 10^{p-1}10^{d-p-1}$. By changing the first $p-1$ bits and the $(p+1)^{\rm st}$ bit of $x_2$ we arrive to the vertex $x_3 = 01^{p-2}0^{d-p+1}$. Since $p-2$ is odd, we can write $p = p_1 + p_2$, where $p_1 - p_2 = 1$. Let $x_4$ be a neighbor of $x_3$ in $Q\pow{p}$ obtained from $x_3$ by changing $p_2$ of its $1$s into $0$s, and hence $p-p_2$ zero bits of $x_3$ into $1$s. In this way $x_4$ contains $p_1 + (p-p_2) = p+1$ bits equal to $1$. Now $x_4$ is in $Q_d\pow{p}$ adjacent to $p+1$ vertices each of which has exactly one bit $1$.  
\end{proof}

\section{Exact distance graphs of hypercubes}
\label{sec:hypercubes}

In this section we first describe the structure of exact distance graphs of hypercubes by showing that they contains copies of some generalized Johnson graphs.
Afterward, we combine known results and new ones about the chromatic number of generalized Johnson graphs to derive upper bounds for the chromatic number of some exact distance graphs of hypercubes.

The {\em generalized Johnson graph} $J(n,k,i)$ (where $i\le k\le n$) is the graph with the set $\{A\subseteq \{1,\ldots,n\}:\ |A|=k\}$ and edge set $\{AB :\ |A\cap B| = i\}$. The family of generalized Johnson graphs includes Kneser graphs $K(n,k)=J(n,k,0)$ (which themselves include the odd graphs $J(2k+1,k,0)$) and the Johnson graphs $J(n,k,k-1)$)~\cite{AGON}.

\subsection{On the structure of exact distance graphs of hypercubes}

For even distance, the structure of the exact distance graph of the hypercube is known.
\begin{proposition}{\rm (\cite{ZIEGLER})}
\label{prop:q2}
$Q_n\pow{2p}=2 (Q_{n-1}\pow{2p}\uplus Q_{n-1}\pow{2p-1})$. 
\end{proposition}

For odd distance $n-1$ we prove the existence of the following isomorphism.

\begin{proposition}
For every positive even integer $n$, $Q_{n}\pow{n-1} \cong Q_{n}$.
\end{proposition}
\begin{proof}
For a vertex $x$ of $Q_n\pow{n-1}$, we denote by $x_{i,i+1}$, for $i\in\{1,\ldots,n-1\}$, the concatenation of the $i$th bit and $(i+1)$th bit of $x$. We say that $x_{i,i+1}$ is an odd word if $x_{i,i+1}\in \{01,10\}$, and otherwise $x_{i,i+1}$ is an even word (i.e., when $x_{i,i+1}\in \{00,11\}$). Next, if 
$\{x_{1,2},x_{3,4},\ldots, x_{n-1,n}\}$ contains an even number of odd words, then $x$ is said to be of type A. Otherwise, $x$ is said to be of type B.
We set the following function $f$, for $i\in\{0,\ldots,(n-2)/2\}$ from $\{0,1\}^n$ to $\{0,1\}^n$:
$$f(x)_{2i+1,2i+2}= \left\{
    \begin{array}{ll}
        x_{2i+1,2i+2}, & \text{ if } x_{2i+1,2i+2} \text{ is even and $x$ is of type A}; \\[2mm]
        \overline{x_{2i+1,2i+2}},& \text{ if } x_{2i+1,2i+2} \text{ is odd and $x$ is of type A};\\[2mm]
         \overline{x_{2i+1,2i+2}},& \text{ if } x_{2i+1,2i+2} \text{ is even and $x$ is of type B};\\[2mm]
         x_{2i+1,2i+2}, & \text{ otherwise}.\\
    \end{array}
\right.
$$

We first prove that $f$ is a bijective and, afterwards, that $f$ is an isomorphism between $Q_{n}\pow{n-1}$ and $Q_{n}$.
First, since $f(x)$ is of type A if and only if $x$ is of type A, it can be easily noticed that $x\neq x'$, for $x,x'\in V(Q_{n}\pow{n-1})$, implies $f(x)\neq f(x')$. Also, for every vertex $y$ of $Q_n$, there exists $x\in V(Q_{n}\pow{n-1})$ such that $f(x)=y$. Thus, $f$ is bijective.

Second, since $n-1$ is odd, every adjacent vertices $x$ and $x'$ in $Q_{n}\pow{n-1}$ are in different classes ( $Q_{n}\pow{n-1}$ is bipartite). Suppose that $x$ and $x'$ differ in $n-1$ bits, i.e., have exactly one common bit $x_{k}$. Suppose that $x_{2i+1, 2i+2}$ contains the bit $x_k$. It can be easily observed that for each $j\in \{0,\ldots,(n-2)/2\}\setminus \{i\}$, $x_{2j+1,2j+2}$ is an even word if and only if $x'_{2j+1,2j+2}$ is an even word. Moreover, $x_{2i+1, 2i+2}$ is an even word if and only if $x'_{2i+1, 2i+2}$ is an odd word. Consequently, for each $j\in \{0,\ldots,(n-2)/2\}\setminus\{i\}$, $f(x)_{2j+1,2j+2}=f(x')_{2j+1,2j+2}$ and since $f(x)_{2i+1,2i+2}$ and $f(x')_{2i+1,2i+2}$ have exactly one common bit, $f(x)$ and $f(x')$ have $n-1$ common bits and, consequently, are adjacent in $Q_{n}$. Finally, if $x$ and $x'$ are not adjacent in $Q_{n}\pow{n-1}$, then $f(x)$ and $f(x')$ are not adjacent in $Q_n$, since both $Q_{n}\pow{n-1}$ and $Q_{n}$ are $n$-regular. Thus, $f$ is an isomorphism.
\end{proof}

The following isomorphism is well known, cf.~\cite{AGON}.

\begin{proposition}
\label{isojohnson}
If $n$, $k$, and $i$ are positive integers, then $J(n,k,i)\cong J(n,n-k,n-2k+i)$.
\end{proposition}

The set of vertices of $Q_n$ having exactly $j$ bits $1$ will be denoted by $L_j^n$, or shortly $L_j$ when the hypercube $Q_n$ is understood from context.

\begin{proposition}\label{hypeven}
For every integer $n$ and even integer $p$, $p\le n$, the exact distance graph $Q_{n}\pow{p} $ contains $J(n,i,i-p/2)$ as an induced subgraph, for each $i\in\{p/2,\ldots,n -p/2 \}$. Moreover, all these induced subgraphs are pairwise vertex disjoint in $Q_{n}\pow{p}$.
\end{proposition}

\begin{proof}
By changing $p/2$ bits $1$ and $p/2$ bits $0$ from a vertex of $L_{i}$, we obtain another vertex from $L_{i}$ with $i-p/2$ common bits $1$.
If we change $k$ bits $1$, with $k>p/2$ or $k<p/2$ from a vertex of $L_{i}$, then we obtain a vertex of $V(G)\setminus L_i$.
Thus, the vertices of $L_{i}^{n}$ induce the graph $J(n,i,i-p/2)$.
\end{proof}

\begin{remark}\label{hyprem}
For every integer $n$ and even integer $p$, where $p\le n$, the subgraph induced by $ \cup_{0\le j\le \lfloor n/2\rfloor} L_{2j}^{n}$ and the subgraph induced by $\cup_{0\le j\le \lfloor (n-1)/2\rfloor} L_{2j+1}^{n}$ are the two isomorphic connected components of $Q_{n}\pow{p}$. 
\end{remark}

This remark follows from two facts.
First, when $p$ is even there is no edges between a vertex containing an even number of bits $1$ and a vertex containing an odd number of bits $1$ (by parity).
Second, by inverting the bits $0$ and $1$, we have a trivial isomorphism between $ \cup_{0\le j\le \lfloor n/2\rfloor} L_{2j}^{n}$ and $\cup_{0\le j\le \lfloor (n-1)/2\rfloor} L_{2j+1}^{n}$.

\subsection{Colorings of the generalized Johnson graphs}

The determination of the chromatic number of Kneser graphs is a classical result of graph theory~\cite{BARANY,LOVASZ, mat-04}.
\begin{theorem}{\rm (\cite{BARANY,LOVASZ})}
\label{thm:lovasz}
For any integers $n$ and $k<n/2$, $\chi(J(n,k,0))=n-2k+2$.
\end{theorem}

On the other hand, it is not known much about the chromatic number of generalized Johnson graphs and related graph classes. We state a few known bounds and values in this area.

\begin{theorem}{\rm (\cite{BOBU})}
\label{thm:bobu}
We have $\chi(J(6,3,1))=6$ and $\chi(J(8,4,1))=5$.
\end{theorem}

\begin{theorem} {\rm (\cite{BOBU})}
For any positive integers $n$ and $i<n/2$, $i+2 \le \chi(J(n,n/2,i))\le 2 \binom{2i+2}{i+1} $.
\end{theorem}

This latter result was recently extended and improved by Balogh, Cherkashin and Kiselev~\cite{BALOGH} who presented a upper bound which is quadratic on $i$ even for the generalized Kneser graph.

We define the {\em generalized Kneser graph} $K(n,k,i)$, where $i\le k\le n$, as the graph with vertex set $\{A\subseteq \{1,\ldots,n\}:\ |A|=k\}$ and edge set $\{AB :\ |A\cap B|\le i\}$. 
For homogeneity reasons, the generalized Kneser graphs are defined slightly differently than in \cite{BALOGH,JARAFI} (there is a shift for the third parameter).
Note that the generalized Johnson graph $J(n,k,i)$ is a subgraph of  $K(n,k,i)$. Consequently, $\chi(J(n, k,i))\le \chi(K(n, k,i))$.
The following are known results about the chromatic number of generalized Kneser graphs.

\begin{theorem}{\rm (\cite{JARAFI})}
For every positive integers $n$, $k$ and $i$, $\chi(K(n, k,i))\le \binom{ n-2k+2(i+1) }{i+1}$.
\end{theorem}

\begin{theorem}{\rm (\cite{JARAFI})}
\label{thm:jarafi}
For any $0< i+1< k<n$, we have $\chi(K(n+2,k+1,i))\le \chi(K(n,k,i)))$.
In particular, for $k\ge 3$:
\begin{itemize}
\item $\chi(K(2k,k,1))\le \chi(K(6,3,1))\le 6$;
\item $\chi(K(2k+1,k,1))\le \chi(K(7,3,1))\le 9$;
\item $\chi(K(2k+2,k,i))\le \chi(K(8,3,1))$. 
\end{itemize}
\end{theorem}

In the following proposition, we give an upper bound on the chromatic number of $K(8,3,1)$.
Note that by Theorem~\ref{thm:jarafi} this upper bound implies the same upper bound on $\chi(J(2k+2,k,i))$, for $0< i+1< k<n$ and $k\ge 3$.

\begin{proposition}
\label{prp:newbound}
We have $\chi(K(8,3,1))\le 12$.
\end{proposition}
\begin{proof}
We claim that the mapping $c:V(K(8,3,1))\rightarrow \{1,\ldots, 12\}$, defined by 
$$c(A)= \left\{
    \begin{array}{cl}
        i, & \text{ if } \{2i-1,2i\}\subseteq A,\ 1\le i\le 4;\\
        5, &\text{ if } A=\{1,4,j\},\ j\in\{5,6,7,8\};\\
        6, &\text{ if } A=\{2,3,j\},\ j\in\{5,6,7,8\};\\
        7, &\text{ if } A=\{j,5,8\},\ j\in\{1,2,3,4\}; \\
        8, &\text{ if } A=\{j,6,7\},\ j\in\{1,2,3,4\}; \\
        9, &\text{ if } A\subseteq\{1,3,5,7\}; \\
        10, & \text{ if } A\subseteq\{1,3,6,8\}; \\
        11, &\text{ if } A\subseteq\{2,4,5,7\}; \\
        12, &\text{ otherwise (if } A\subseteq\{2,4,6,8\});
    \end{array}
\right.
$$
is a proper coloring of $K(8,3,1)$ with twelve colors.

We start by proving that $c$ is well defined (i.e., every vertex $A$ of $K(8,3,1)$ receives a unique color by the above definition). First note that at least two elements of $A$ are either in $\{1,2,3,4\}$ or in $\{5,6,7,8\}$. Suppose, without loss of generality, that at least two elements of $A$ are in $\{1,2,3,4\}$.
If $\{1,2\}\subseteq A$, $\{3,4\}\subseteq A$, $\{1,4\}\subseteq A$ or $\{2,3\}\subseteq A$, then $c(A)\in\{1,2,3,4,5,6\}$. Consequently, by excluding this case, either $\{1,3\}\subseteq A$ or $\{2,4\}\subseteq A$, and consequently $A$ has a color among $\{9,10,11,12\}$.

Now, we prove that for any two adjacent vertices $A$ and $B$ of $K(8,3,1)$, we have $c(A)\neq c(B)$. If $c(A)=c(B)$ and $c(A)\in\{1,2,3,4,5,6,7,8\}$, then $A$ and $B$ have two common elements and are thus not adjacent.
Since any two vertices, which are subsets of a set of size $4$, have two elements in common, we infer that $c(A)=c(B)$ and $c(A)\in\{9,10,11,12\}$ implies that $A$ and $B$ are not adjacent.
\end{proof}

\subsection{Colorings of exact distance graphs of hypercubes}
Bounds or exact values are known for the chromatic number of exact distance-$p$ graph of the hypercube. We skip mentioning numerous results about the chromatic number of $Q_n\pow{2}$ since by Proposition~\ref{prop:q2} it is in relation with the chromatic number of the second power of the hypercube.

\begin{theorem}{\rm (\cite{PAYAN,ZIEGLER})}
If $n$ is an odd integer, then $\chi(Q_{n}\pow{n-1})=  4.$
\end{theorem}

\begin{theorem}{\rm (\cite{ZIEGLER})}
We have $\chi(Q_{6}\pow{4})= 7$, $\chi(Q_{7}\pow{4})= 8$,  $\chi(Q_{8}\pow{4})= 8$, $\chi(Q_{8}\pow{6})\le 8$ and $\chi(Q_{9}\pow{6})\le 16$.
\end{theorem}

\begin{theorem}{\rm ([\cite{FWFULIXI})}
We have $\chi(Q_{n}\pow{d})\le 2^{\lceil \log_2(1+\binom{n-1}{d-1}) \rceil}$.
\end{theorem}

Table~\ref{sum} illustrates the upper bounds obtained in this section for small values of $n$. It can be observed that we have improved the results from Ziegler on $Q_{8}\pow{6}$ and $Q_{9}\pow{6}$. 
The lower bounds from Table~\ref{sum} are obtained by using Theorem~\ref{thm:lovasz} (by Proposition \ref{hypeven}, $Q_{n}\pow{p}$ contains $J(n,p/2,0))$ as induced graph).

\begin{table}
   \caption{\label{sum} Bounds on $\chi(Q_{n}\pow{p})$. Bold numbers represent  exact values, a pair $a$-$b$ represents a lower bound and an upper bound on $\chi(Q_{n}\pow{p})$.}
\begin{center}
\begin{tabular}{|c|c|c|c|c|}
\hline
$n \backslash p$ & \ \ \ \ \ \ 4 \ \ \ \ \ \ & \ \ \ \ \ \ 6 \ \ \ \ \ \ & \ \ \ \ \ \ 8 \ \ \ \ \ \ &  \ \ \ \ \ \ 10 \ \ \ \ \ \ \\ 
\hline
6 & \textbf{7} \cite{PAYAN} & \textbf{2} & & \\
7 & \textbf{8} \cite{ZIEGLER} & \textbf{4} \cite{ZIEGLER} & &  \\
 8 & \textbf{8} \cite{ZIEGLER} & 4-7 & \textbf{2} &     \\
  9 & \textbf{8} \cite{ZIEGLER} & 5-15 & 4-8 &    \\
  10 &   & 6-26 & 5-15 & \textbf{2}  \\
  \hline
\end{tabular}
\end{center}
\end{table}

Using the structural tools of the previous subsection, we derive new results about the chromatic number of $Q_{n}\pow{d}$ for $n-d \le 4$ improving some of the above results. (The situation when $n-d>4$ could be handled in a similar way.)
Recall that when $d$ is odd, $Q_{n}\pow{d}$ is bipartite, hence in the following we only consider even $d$.

\begin{theorem}
\label{thm:n_n-2}
If $n\ge 4$ is an even positive integer, then $$\chi(Q_{n}\pow{n-2})\le \chi(J(n,n/2,1))+2.$$
\end{theorem}
\begin{proof}
Note that, by Remark~\ref{hyprem}, one connected component of $Q_{n}\pow{n-2}$ contains the vertices of both $L_{(n-2)/2}$  and $L_{(n+2)/2}$ and the other one the vertices of $L_{n/2}$.  

It is possible to color the vertices of $L_{n/2}$ with $\chi(J(n,n/2,1))$ colors.
Note that if two vertices $u$ and $v$ differ in exactly $n-2$ bits, $u\in L_{i}$, for $i\le (n-4)/2$, then it implies $v\in L_{j}$ for  $j\ge n/2$.
Consequently there is no edge between two vertices with less than $(n-4)/2$ bits $1$. Similarly, there is no edge between two vertices with more than $(n+4)/2$ bits $1$. Finally, we use two new colors to color all the vertices in $L_i$, for $i\le (n-4)/2$ with the same color and to color all the vertices in $L_i$, for $i\ge (n+4)/2$ with the same color.
\end{proof}

\begin{corollary}
If $n\ge 4$ is an even positive integer, then $\chi(Q_{n}\pow{n-2})\le  8$. In addition, $\chi(Q_{8}\pow{6})\le 7$.
\end{corollary}
\begin{proof}
The first assertion follows by combining Theorem~\ref{thm:n_n-2} (left bound) with Theorem~\ref{thm:jarafi}. The second assertion follows by combining Theorem~\ref{thm:n_n-2} (right bound) with Theorem~\ref{thm:bobu}.
\end{proof}

\begin{theorem}
\label{thm:n_n-3}
If $n\ge 5$ is an odd positive integer, then 
$$\chi(Q_{n}\pow{n-3})\le  \chi(J(n,(n-3)/2,0))+\chi(J(n,(n-1)/2,1))+1\le  \chi(K(7,3,1))+6.$$
\end{theorem}
\begin{proof}
Note that one connected component of $Q_{n}\pow{n-3}$ contains the vertices of both $L_{(n-3)/2}$ and $L_{(n+1)/2}$ and the other one the vertices of  both $L_{(n-1)/2}$ and $L_{(n+3)/2}$.
By Proposition \ref{hypeven} and its proof, the vertices from $L_{(n-3)/2}$ induce the graph $J(n,(n-3)/2,0)$ and the vertices from $L_{(n+1)/2}$ induce the graph $J(n,(n-1)/2,1)$. Also, by Proposition \ref{isojohnson}, we have $J(n,(n-3)/2,0)\cong J(n,(n+3)/2,3)$ and $J(n,(n-1)/2,1)\cong J(n,(n+1)/2,2)$.

It is possible to color the vertices of $L_{(n-3)/2}$ with $\chi(J(n,(n-3)/2,0))$ colors and to color the vertices of $L_{(n+1)/2}$ with $\chi(J(n,(n-1)/2,1))$ colors.

Note that for every two vertices $u$ and $v$ differing in exactly $n-3$ bits, $u\in L_{i}$, for $i\le (n-3)/2$, we have $v\in L_{j}$ for  $j\ge(n-3)/2$.
Consequently there is no edge between two vertices with less than $(n-3)/2$ bits $1$. Similarly, there is no edge between two vertices with more than $(n+3)/2$ bits $1$. Thus we can use just one new color for all the vertices in $L_i$, for $i> (n+3)/2$. Finally, note that no vertex of $L_{(n-3)/2}$ is adjacent to a vertex of $L_{i}$ for $i<(n-3)/2$, hence it is possible to re-use a color used for $L_{(n-3)/2}$ to color all the vertices in $L_i$, for $i< (n-3)/2$. 
\end{proof}

Combining Theorem~\ref{thm:n_n-3} with Theorem~\ref{thm:jarafi} we infer the following bound. 

\begin{corollary}
If $n\ge 5$ is an odd positive integer, then $\chi(Q_{n}\pow{n-3})\le  15$.
\end{corollary}

\begin{theorem}
\label{thm:n_n-4}
If $n\ge 6$ is an even positive integer, then 
$$\begin{array}{lcl} 
\chi(Q_{n}\pow{n-4}) &\le&  \min\{2\chi(J(n,(n-4)/2,0))+\chi(J(n,n/2,2)),  2\chi(J(n,(n-2)/2,1))+2 \} \\[3mm]
&\le& 2\chi(K(8,3,1))+2.
\end{array}$$
\end{theorem}
\begin{proof}
Note that one connected component of $Q_{n}\pow{n-4}$ contains the vertices of both $L_{(n-4)/2}$, $L_{n/2}$ and $L_{(n+4)/2}$ and the other one the vertices of both $L_{(n-2)/2}$ and $L_{(n+2)/2}$.  
First, we begin by proving that $\chi(Q_{n}\pow{n-4})\le 2\chi(J(n,(n-4)/2,0))+\chi(J(n,n/2,2))$.
By Proposition \ref{hypeven}, the vertices from $L_{(n-4)/2}$ induce the graph $J(n,(n-4)/2,0)$, the vertices from $L_{(n-2)/2}$ induce the graph $J(n,(n-2)/2,1)$ and the vertices from $L_{n/2}$ induce the graph $J(n,n/2,2)$. By Proposition \ref{isojohnson}, $J(n,(n-4)/2,0)\cong J(n,(n+4)/2,4)$ and $J(n,(n-2)/2,1)\cong J(n,(n+2)/2,3)$. Consequently, it is possible to color the vertices of $L_{(n-4)/2}$, $L_{n/2}$, and $L_{(n+4)/2}$ with $2\chi(J(n,(n-4)/2,0))+\chi(J(n,n/2,2))$ colors.
Note that for vertices $u$ and $v$ differing in exactly $n-4$ bits, $u\in L_{i}$, for $i\le (n-4)/2$, we have $v\in L_{j}$ for  $j\ge(n-4)/2$.
Consequently there is no edge between two vertices with less than $(n-4)/2$ bits $1$. Similarly, there is no edge between two vertices with more than $(n+4)/2$ bits $1$. Finally, it is possible to re-use a color used for $L_{(n-4)/2}$ to color all the vertices in $L_i$, for $i< (n-2)/2$  and to re-use a color used for $L_{(n+4)/2}$ to color all the vertices in $L_i$, for $i> (n+2)/2$.

Second, we prove that $\chi(Q_{n}\pow{n-4})\le 2\chi(J(n,(n-2)/2,1))+2$.
It is possible to color the vertices of $L_{(n-2)/2}$ with $\chi(J(n,(n-2)/2,1))$ colors.
Note that for every two vertices $u$ and $v$ differing in exactly $n-6$ bits, $u\in L_{i}$, for $i\le (n-6)/2$, we have $v\in L_{j}$ for $j\ge (n-3)/2$.
Consequently there is no edge between two vertices with less than $(n-6)/2$ bits $1$. Similarly, there is no edge between two vertices with more than $(n+6)/2$ bits $1$. Finally, we use two new colors to color all the vertices in $L_i$, for $i\le (n-6)/2$, with the same color and to color all the vertices in $L_i$, for $i\ge (n+4)/2$ with the same color.
\end{proof}

Combining Proposition~\ref{prp:newbound} with Theorem~\ref{thm:n_n-4} we get

\begin{corollary}
If $n\ge 6$ is an even positive integer, then $\chi(Q_{n}\pow{n-4})\le  26$.
\end{corollary}

\section*{Acknowledgement}

This work was performed with the financial support of the bilateral project "Distance-constrained and game colorings of graph products" (BI-FR/18-19-Proteus-011).

B.B. and S.K. acknowledge the financial support from the Slovenian Research Agency (research core funding No.\ P1-0297 and project Contemporary invariants in graphs No.\ J1-9109).

\end{document}